\documentclass{llncs}

\usepackage{proof}
\usepackage{amssymb,amsfonts,amsmath,mathrsfs}
\usepackage[pdftex]{graphicx}

\newcommand{\SL}{\mathop{/}}
\newcommand{\BS}{\mathop{\backslash}}
\newcommand{\U}{\mathbf{1}}
\newcommand{\Morrill}{\mathbf{Db}\boldsymbol{!}_\mathbf{b}}
\newcommand{\MorrillB}{\mathbf{Db}\boldsymbol{!}}
\newcommand{\CC}{\boldsymbol{!}_\mathbf{b} \mathbf{L}^{\!\U}}
\newcommand{\CCbfp}{\boldsymbol{!} \mathbf{L}^{\!\U}}
\newcommand{\CCwk}{\boldsymbol{!}_\mathbf{w} \mathbf{L}^{\!\U}}
\newcommand{\LL}{\mathbf{L}^{\!\U}}

\newcommand{\BFP}[1]{\text{\sc bfp}(#1)}

\newcommand{\NMod}{[]^{-1}}
\newcommand{\PMod}{\langle\rangle}

\newcommand{\PERM}{\mathrm{perm}}
\newcommand{\CONTRb}{\mathrm{contr}_\mathbf{b}}
\newcommand{\CONTR}{\mathrm{contr}}
\newcommand{\WEAK}{\mathrm{weak}}
\newcommand{\SHOOT}{\mathrm{inst}}
\newcommand{\CUT}{\mathrm{cut}}

\newcommand{\Bc}{\mathcal{B}}
\newcommand{\Tc}{\mathcal{T}}

\newcommand{\Der}{\mathscr{D}}

\newcommand{\der}[2]{\begin{matrix} \mbox{\includegraphics{derivtree.pdf}} 
\vspace*{-5em}\\ #1 \\[3.5em] #2\end{matrix}}

\begin{document}

\title{Undecidability of the Lambek Calculus with 
Subexponential and Bracket Modalities}

\author{Max Kanovich\inst{1} \and
Stepan Kuznetsov\inst{2} \and Andre Scedrov\inst{3}}
\tocauthor{Max Kanovich, Stepan Kuznetsov, and Andre Scedrov}
\institute{
University College London and \\
National Research University Higher School of Economics (Moscow);\\
\email{m.kanovich@qmul.ac.uk}
\and
Steklov Mathematical Institute, RAS (Moscow);\\
\email{sk@mi.ras.ru}
\and
University of Pennsylvania (Philadelphia) and\\
National Research University Higher School of
Economics (Moscow); \\
\email{scedrov@math.upenn.edu}}

\maketitle



\begin{abstract}
The Lambek calculus is a well-known logical formalism for modelling natural 
language syntax. The original calculus covered a substantial number 
of intricate natural language phenomena, but only those restricted to 
the context-free setting. In order to address more subtle linguistic issues, 
the Lambek calculus has been extended in various ways. In particular, 
Morrill and Valent\'{\i}n (2015) introduce an extension with so-called exponential 
and bracket modalities. Their extension is based on a non-standard contraction rule 
for the exponential that interacts with the bracket structure in an intricate way. 
The standard contraction rule is not admissible in this calculus.
In this paper we prove undecidability of the derivability problem in their calculus. 
We also investigate restricted decidable fragments considered by Morrill and 
Valentin and we show that these fragments belong to the NP class. 
\end{abstract}

\section{Linguistic Introduction}

The Lambek calculus~\cite{Lambek58} is a substructural, non-commutative logical system
(a variant of linear logic~\cite{Girard} in its intuitionistic non-commutative version~\cite{Abrusci})
that serves as the logical base for categorial grammars, a formalism that aims to describe
natural language by means of logical derivability (see Buszkowski~\cite{BuszkoTLiG}, 
Carpenter~\cite{Carpenter}, Morrill~\cite{MorrillBook}, Moot and Retor\'{e}~\cite{MootRetore}, {\it etc}). 
The idea of categorial grammar goes back to works of Ajdukiewicz~\cite{Ajdukiewicz} and Bar-Hillel~\cite{BarHillel},
and afterwards it developed into several closely related frameworks, including combinatory categorial grammars (CCG, Steedman~\cite{Steedman}),
categorial dependency grammars (CDG, Dikovsky and Dekhtyar~\cite{DikovskyDekhtyar}), and Lambek categorial grammars.
A categorial grammar assigns syntactic categories (types) to words of the language.
In the Lambek setting, types are constructed using two division operations, $\BS$ and $\SL$, and the product,~$\cdot$.
Intuitively, $A \BS B$ denotes the type of a syntactic object that lacks something of type $A$
on the left side to become an object of type $B$; $B \SL A$ is symmetric; the product stands
for concatenation. The Lambek calculus provides a system of rules for reasoning about syntactic
types. 

For a quick example, consider the sentence {\sl ``John loves Mary.''} Let {\sl ``John''} and
{\sl ``Mary''} be of type $N$ (noun), and {\sl ``loves''} receive the type $(N \BS S) \SL N$ of
the transitive verb: it takes a noun from the left and a noun from the right, yielding a sentence, $S$.
This sentence is judged as a grammatical one, because $N, (N \BS S) \SL N, N \to S$ is a
theorem in the Lambek calculus (and even in the Ajdukiewicz -- Bar-Hillel logic for basic categorial grammars).

The Lambek calculus is capable of handling more complicated situations, including 
dependent clauses: {\sl ``the girl whom John loves''}, parsed as $N$ using the following
types: $N \SL CN, CN, (CN \BS CN) \SL (S \SL N), N,  (N \BS S) \SL N \to N$ (here $CN$ stands
for ``common noun,'' a noun without an article), and coordination:
{\sl ``John loves Mary and Pete loves Kate,''} where {\sl ``and''} is $(S \BS S) \SL S$.

There are, however, even more sophisticated cases for which the pure Lambek calculus is known to be insufficient
(see, for example,~\cite{MorrillBook}\cite{MootRetore}).
On the one hand, for a noun phrase like {\sl ``the girl whom John met yesterday''} it is problematic
to find a correct type for {\sl ``whom,''} since the dependent clause {\sl ``John met yesterday''} expects the
lacking noun
({\sl ``John met ... yesterday''}; the ``...'' place is called {\em gap})
in the middle, and therefore is neither of type $S \SL N$ nor of type $N \BS S$. This phenomenon is
called {\em medial extraction.} On the other hand, the grammar sketched above generates,
for example, {\sl *``the girl whom John loves Mary and Pete loves.''} The asterisk indicates ungrammaticality---but
 {\sl ``John loves Mary and Pete loves''} is yet of 
  type $S \SL N$. 
  To avoid this, one needs to block extraction
from certain syntactic structures ({\em e.g.,} compound sentences), called {\em islands}~\cite{Ross}\cite{MorrillBook}.

These issues can be addressed by extending the Lambek calculus with extra connectives (that allow
to derive more theorems) and also with a more sophisticated syntactic structure (that allows blocking
unwanted derivations). In the next section, we follow Morrill and Valent\'{\i}n~\cite{MorrillBook}\cite{MorVal}
and define an extension of the Lambek calculus with a subexponential modality (allows medial and also
so-called parasitic extraction) and brackets (for creating islands). 

\section{Logical Introduction} 
\label{S:Calculus}

In order to block ungrammatical extractions, such as discussed above, Morrill~\cite{Morrill1992} and Moortgat~\cite{Moortgat1996} introduce
an extension of the Lambek calculus with brackets that create islands. For the second issue, medial extraction,
Morrill and Valent\'{\i}n~\cite{Barry}\cite{MorVal}  
suggest using a  modality which they call ``exponential,'' in the spirit of Girard's exponential in linear logic~\cite{Girard}.
We rather use the term ``subexponential,'' which is due to Nigam and Miller~\cite{NigamMiller}, since this modality allows only some
of the structural rules (permutation and contraction, but not weakening). The difference from~\cite{NigamMiller},
however, is in the non-commutativity of the whole system 
 and the non-standard nature of the contraction rule.

We consider
$\CC$, the Lambek calculus with the unit constant~\cite{LambekUnit}, 
brackets, and a subexponential controlled by rules from~\cite{MorVal}. The calculus $\CC$ is a conservative
fragment of the $\Morrill$ system by Morrill and Valent\'{\i}n~\cite{MorVal}.

Due to brackets, the syntax of $\CC$ is more involved than the syntax of a standard sequent calculus. 
Derivable objects are sequents of the form $\Pi \to A$. 
The antecedent $\Pi$ is a structure called
{\em meta-formula} (or {\em configuration});
the succedent $A$ is a formula. Meta-formulae are built from formulae (types) using two
metasyntactic operators: comma and brackets. The succedent $A$ is a formula. Formulae, in their turn,
are built from {\em primitive types} (variables) $p_1, p_2, \ldots$ and the unit constant $\U$ 
using the Lambek's binary connectives: $\BS$, $\SL$, and $\cdot$, and three unary connectives, $\PMod$, $\NMod$,
and ${!}$. The first two unary connectives operate brackets; the last one is the subexponential used
for medial extraction.

Meta-formulae are denoted by capital Greek letters; $\Delta(\Gamma)$ stands for $\Delta$ with a designated occurrence
of a meta-formula (in particular, formula) $\Gamma$. Meta-formulae are allowed to be empty;
the empty meta-formula is denoted by $\Lambda$.

The axioms of $\CC$ are $A \to A$ and $\Lambda \to \U$, and the rules are as follows:
$$
\infer[(\SL\to)]{\Delta (C \SL B, \Gamma) \to D}
{\Gamma \to B & \Delta (C) \to D}
\qquad
\infer[(\to\SL)]{\Gamma \to C \SL B}{\Gamma, B \to C}
\qquad
\infer[(\cdot\to)]{\Delta (A \cdot B) \to D}
{\Delta (A, B) \to D}
$$
$$
\infer[(\BS\to)]{\Delta (\Gamma, A \BS C) \to D}
{\Gamma \to A & \Delta (C) \to D}
\qquad
\infer[(\to\BS)]{\Gamma \to A \BS C}{A, \Gamma \to C}
\qquad
\infer[(\to\cdot)]
{\Gamma_1, \Gamma_2 \to A \cdot B}
{\Gamma_1 \to A & \Gamma_2 \to B}
$$
$$
\infer[(\U\to)]{\Delta(\U) \to A}{\Delta (\Lambda) \to A}
\qquad
\infer[(\PMod\to)]
{\Delta(\PMod A) \to C}{\Delta([A]) \to C}
\quad
\infer[(\to\PMod)]
{[\Pi] \to \PMod A}{\Pi \to A}
$$  
$$
\infer[({!}\to)]{\Gamma({!}A) \to B}{\Gamma (A) \to B}
\qquad
\infer[(\NMod\!\to)]
{\Delta([\NMod A]) \to C}{\Delta(A) \to C}
\quad
\infer[(\to\NMod)]
{\Pi \to \NMod A}{[\Pi] \to A}
$$
$$
\infer[(\to{!})]{{!}A_1, \dots, {!}A_n \to {!}A}
{{!}A_1, \dots, {!}A_n \to A}
\qquad
\infer[(\CONTRb)]{\Delta ({!}A_1, \dots, {!}A_n, \Gamma ) \to B}
{\Delta ({!}A_1, \dots, {!}A_n, [{!}A_1, \dots, {!}A_n, \Gamma] ) \to B}
$$
$$
\infer[(\PERM_1)]{\Delta (\Gamma, {!}A) \to B}
{\Delta ({!}A, \Gamma) \to B}
\quad
\infer[(\PERM_2)]{\Delta ({!}A, \Gamma) \to B}
{\Delta (\Gamma, {!}A) \to B}
\quad
\infer[(\CUT)]{\Delta (\Pi) \to C}{\Pi \to A & \Delta(A) \to C}
$$

The permutation rules $(\PERM_{1,2})$ for ${!}$ allow medial extraction. The relative pronoun
{\sl ``whom''} now receives the type $(CN \BS CN) \SL (S \SL {!}N)$, and the noun phrase
{\sl ``the girl whom John met yesterday''} now becomes derivable (the type for
{\sl ``yesterday''} is $(N \BS S) \BS (N \BS S)$, modifier of verb phrase):
$$\scriptsize
\infer{N \SL CN, CN, (CN \BS CN) \SL (S \SL {!}N), N, (N \BS S) \SL N, (N \BS S) \BS (N \BS S) \to N}
{\infer{N, (N \BS S) \SL N, (N \BS S) \BS (N \BS S) \to S \SL {!} N}
{\infer{N, (N \BS S) \SL N, (N \BS S) \BS (N \BS S), {!}N \to S}
{\infer{N, (N \BS S) \SL N, {!}N, (N \BS S) \BS (N \BS S) \to S}
{\infer{N, (N \BS S) \SL N, N, (N \BS S) \BS (N \BS S) \to S}
{N \to N & \infer{N, N \BS S, (N \BS S) \BS (N \BS S) \to S}
{N \BS S \to N \BS S & \infer{N, N \BS S \to S}{N \to N & S \to S}}}
}}} &
\infer{N \SL CN, CN, CN \BS CN \to N}{\infer{CN, CN \BS CN \to CN}{CN \to CN & CN \to CN} & N \to N}}
$$
The permutation rule puts ${!}N$ to the correct place ({\sl ``John met ... yesterday''}).

For brackets, consider the following ungrammatical example: {\sl *``the book which John laughed without reading.''}
In the original Lambek calculus, it would be generated by the following derivable sequent:
\begin{center}
\scriptsize
$
N \SL CN, CN, (CN \BS CN) \SL (S \SL N), N, N \BS S, ((N \BS S) \BS (N \BS S)) \SL (N \BS S), (N \BS S) \SL N \to N.
$
\end{center}
In the grammar with brackets, however, {\sl ``without''} receives the syntactic type $\NMod ((N \BS S) \BS (N \BS S)) \SL (N \BS S)$,
making the {\sl without}-clause an island that cannot be penetrated by extraction. Thus, the following
sequent is not derivable
\begin{center}
\scriptsize
$
N \SL CN, CN, (CN \BS CN) \SL (S \SL N), N, N \BS S, [ \NMod((N \BS S) \BS (N \BS S)) \SL (N \BS S), (N \BS S) \SL N] \to N,
$
\end{center}
and the ungrammatical example gets ruled out.

Finally, the non-standard contraction rule, $(\CONTRb)$, that governs both ${!}$ and brackets, 
was designed for handling a more rare phenomenon called
{\em parasitic extraction.} It appears in examples like {\sl ``the paper that John signed without reading.''} Compare
with the ungrammatical example considered before: now in the dependent clause there are two gaps, and
one of them is inside an island
({\sl ``John signed ... [without reading ...]''}); both gaps are filled with the same ${!}N$: 
$$\scriptsize
\infer{N \SL CN, CN, (CN \BS CN) \SL (S \SL {!}N), N, (N \BS S) \SL N, \NMod((N \BS S) \BS (N \BS S)) \SL (N \BS S), (N \BS S) \SL N \to N}
{\infer{N, (N \BS S) \SL N, \NMod((N \BS S) \BS (N \BS S)) \SL (N \BS S), (N \BS S) \SL N \to S \SL {!}N}
{\infer{N, (N \BS S) \SL N, \NMod((N \BS S) \BS (N \BS S)) \SL (N \BS S), (N \BS S) \SL N, {!}N \to S}
{\infer{N, (N \BS S) \SL N, {!}N, \NMod((N \BS S) \BS (N \BS S)) \SL (N \BS S), (N \BS S) \SL N \to S}
{\infer{N, (N \BS S) \SL N, {!}N, [ {!}N, \NMod((N \BS S) \BS (N \BS S)) \SL (N \BS S), (N \BS S) \SL N] \to S}
{\infer{N, (N \BS S) \SL N, {!}N, [ \NMod((N \BS S) \BS (N \BS S)) \SL (N \BS S), (N \BS S) \SL N, {!}N] \to S}
{\infer{N, (N \BS S) \SL N, N, [ \NMod((N \BS S) \BS (N \BS S)) \SL (N \BS S), (N \BS S) \SL N, N] \to S}
{N \to N &
\infer{N, N \BS S, [ \NMod((N \BS S) \BS (N \BS S)) \SL (N \BS S), (N \BS S) \SL N, N] \to S}
{N \to N & 
\infer{N, N \BS S, [ \NMod ((N \BS S) \BS (N \BS S)) \SL (N \BS S), N \BS S] \to S}
{N \BS S \to N \BS S & 
\infer{N, N \BS S, [ \NMod ((N \BS S) \BS (N \BS S)) ] \to S}
{\infer{N, N \BS S, (N \BS S) \BS (N \BS S) \to S}{N \BS S \to N \BS S & N, N \BS S \to S}}}
}}}}
}}}
&
N \SL CN, CN \to N
}
$$


This construction allows potentially infinite recursion, nesting islands with parasitic extraction.
On the other hand, ungrammatical examples, like
{\sl *``the book that John gave to''} with two gaps outside
islands ({\sl ``John gave ... to ...''}) are not derived with $(\CONTRb)$, 
but can be derived using the contraction rule in the standard, not bracket-aware form:
$
\tfrac{\Delta({!}A, {!}A) \to C}{\Delta({!}A) \to C} (\CONTR)
$.

The system with $(\CONTR)$ instead of $(\CONTRb)$ is a conservative extension of its
fragment without brackets. In an earlier paper~\cite{KanKuzSce2} we show that the latter is undecidable.
For $(\CONTRb)$, however, in the bracket-free fragment there are only permutation rules
for ${!}$, and this fragment is decidable (in fact, it belongs to NP). Therefore, in contrast to~\cite{KanKuzSce2},
the undecidability proof in this paper (Section~\ref{S:undec}) crucially depends on
brackets. 
On the other hand, in~\cite{KanKuzSce2} we've also proved decidability of a fragment of
a calculus with~${!}$, but without brackets. In the calculus considered in this paper, $\CC$,
brackets control the number of $(\CONTRb)$ applications,
whence we are now able to show membership in NP  for a different, broad fragment
of $\CC$ (Section~\ref{S:effective}), which includes brackets. 

It can be easily seen that the calculus with bracket modalities but without ${!}$ also belongs to the NP class.
Moreover, as shown in~\cite{KanKuzMorSce}, there exists even a polynomial algorithm for
deriving formulae of bounded order (connective alternation and bracket nesting depth) in the calculus with brackets but
without~${!}$.
This algorithm uses proof nets, following the ideas of Pentus~\cite{Pentus2010}. As opposed to~\cite{KanKuzMorSce}, 
as we show here, in the presence of~${!}$ the derivability problem is undecidable.

In short,~\cite{KanKuzSce2} is about the calculus with~${!}$, but without brackets; \cite{KanKuzMorSce} is about
the calculus with brackets, but without~${!}$. This paper is about the calculus with both ${!}$ and brackets,
interacting with each other, governed by 
$(\CONTRb)$.


The rest of this paper is organised as follows.
In Section~\ref{S:cutelim} we formulate the cut elimination theorem for $\CC$
and sketch the proof strategy; the detailed proof is placed in Appendix~I.
In Section~\ref{S:nobracket} we define two intermediate calculi used
in our undecidability proof.
In Section~\ref{S:undec} we prove the main result of this paper---the fact
that $\CC$ is undecidable. 
 This solves an open question posed
by Morrill and Valent\'{\i}n~\cite{MorVal} (the other open question from~\cite{MorVal},
undecidability for the case without brackets, is solved in our previous paper~\cite{KanKuzSce2}). In Section~\ref{S:effective} we
consider a practically interesting fragment of $\CC$ for which Morrill and Valent\'{\i}n~\cite{MorVal}
present an exponential time algorithm and strengthen their result by proving an NP upper
bound for the derivability problems in this fragments. 
Section~\ref{S:future} is for conclusion and future research.


\section{Cut Elimination in $\CC$}\label{S:cutelim}


Cut elimination is a natural property that one expects a decent logical system to have.
For example, cut elimination entails the {\em subformula property:} each formula that
appears somewhere in the cut-free derivation is a subformula of the goal sequent.
(Note that for meta-formulae this doesn't hold, since brackets get removed by
applications of some rules, namely, $(\PMod\to)$, $(\to\NMod)$, and $(\CONTRb)$.)

Theorem~\ref{Th:cutelim} is claimed in~\cite{MorVal}, but without a detailed proof.
In this section we give a sketch of the proof strategy; the complete proof is in Appendix~I.

For the original Lambek calculus cut elimination was shown by Lambek~\cite{Lambek58} and 
goes straightforwardly by induction; Moortgat~\cite{Moortgat1996} extended Lambek's proof to
the Lambek calculus with brackets (but without ${!}$). It is well-known, however, that in the 
presence of a contraction rule direct induction doesn't work. Therefore, one needs to use more 
sophisticated cut elimination strategies.

The standard strategy, going back to Gentzen's {\em Hauptsatz}~\cite{Gentzen35},
replaces the cut {\em (Schnitt)} rule with a more general rule called mix
{\em (Mischung).} Mix is a combination of
cut and contraction, and this more general rule can be eliminated by straightforward
induction. For linear logic with the exponential obeying standard rules,
cut elimination is due to Girard~\cite{Girard}; a detailed exposition of the cut elimination procedure using mix is
presented in~\cite[Appendix~A]{LMSS92}.

For $\CC$, however, due to the subtle nature of the contraction rule, 
$(\CONTRb)$, formulating the mix rule is problematic. Therefore, here we follow
another strategy, ``deep cut elimination'' by Bra\"uner and de Paiva~\cite{BraunerPaivaBRICS}\cite{BraunerPaivaCSL};
similar ideas are also used in~\cite{BraunerS5} and~\cite{Eades}. 
As usually, we eliminate one cut, and then proceed by
induction. 

\begin{lemma}\label{Lm:cutelim}
Let $\Delta (\Pi) \to C$ be derived from $\Pi\to A$ and $\Delta(A) \to C$ using the cut rule,
%
and $\Pi \to A$ and $\Delta(A) \to C$ have cut-free derivations $\Der_{\mathrm{left}}$ and $\Der_{\mathrm{right}}$.
Then
$\Delta (\Pi) \to C$ also has a cut-free derivation.
\end{lemma}

We proceed by nested induction on two parameters: (1) the complexity $\kappa$ of the formula $A$ being cut;
(2) the total number $\sigma$ of rule applications in $\Der_\mathrm{left}$ and $\Der_{\mathrm{right}}$. 
Induction goes smoothly for all cases, except the case where the last rule in $\Der_\mathrm{left}$ is
$(\to{!})$ and the last rule in $\Der_{\mathrm{right}}$ is $(\CONTRb)$:
$$
\footnotesize
\infer[(\CUT)]{\Delta({!}\Phi_1, {!}\Pi, {!}\Phi_2, \Gamma) \to C}
{\infer[(\to {!})]{{!}\Pi \to {!}A}{{!}\Pi \to A} &
\infer[(\CONTRb)]{\Delta({!}\Phi_1, {!}A, {!}\Phi_2, \Gamma) \to C}
{\Delta({!}\Phi_1, {!}A, {!}\Phi_2, [{!}\Phi_1, {!}A, {!}\Phi_2, \Gamma]) \to C}}
$$
(Here ${!}\Phi$ stands for ${!}F_1, \dots, {!}F_m$, if $\Phi = F_1, \dots, F_m$.) The na\"{\i}ve attempt, 
$$
\footnotesize
\infer[(\CONTRb)]{\Delta({!}\Phi_1, {!}\Pi, {!}\Phi_2, \Gamma) \to C}
{\infer[(\CUT)]{\Delta ({!}\Phi_1, {!}\Pi, {!}\Phi_2, [{!}\Phi_1, {!}\Pi, {!}\Phi_2, \Gamma]) \to C}
{{!}\Pi \to {!}A & \infer[(\CUT)]{\Delta ({!}\Phi_1, {!}A, {!}\Phi_2, [{!}\Phi_1, {!}A, {!}\Phi_2, \Gamma]) \to C}
{{!}\Pi \to {!}A & \Delta({!}\Phi_1, {!}A, {!}\Phi_2, [{!}\Phi_1, {!}A, {!}\Phi_2, \Gamma]) \to C}}}
$$
fails, since for the lower $(\CUT)$ the $\kappa$ parameter is the same, and $\sigma$ is uncontrolled.
Instead of that, the ``deep'' cut elimination strategy goes inside $\Der_{\mathrm{right}}$ and traces the active ${!}A$ occurrences
up to the applications of $({!}\to)$ which introduced them. Instead of these applications we put $(\CUT)$ with the left
premise, ${!}\Pi \to A$, and replace ${!}A$ with ${!}\Pi$ down the traces. The new $(\CUT)$ instances have a smaller $\kappa$
parameter ($A$ is simpler than ${!}A$) and can be eliminated by induction.

\begin{theorem}\label{Th:cutelim}
Every sequent derivable in $\CC$ has a derivation without $(\CUT)$. 
\end{theorem}




\section{Calculi Without Brackets: $\CCbfp$, $\CCwk$, $\LL$}\label{S:nobracket}

In this section we consider more traditional versions of the Lambek calculus
with ${!}$ that don't include bracket modalities. This is needed as a
technical step in our undecidability proof (Section~\ref{S:undec}).
Types (formulae) of these calculi are built from primitive types
using Lambek's connectives, $\BS$, $\SL$, and $\cdot$, and the
subexponential, ${!}$. Unlike in $\CC$, meta-formulae now are merely linearly ordered sequences of formulae
(possibly empty),
and we can write $\Delta_1, \Pi, \Delta_2$ instead of $\Delta (\Pi)$.

First we define the calculus $\CCbfp$. It includes the standard axioms and rules for  Lambek connectives
and the unit constant---see the rules of $\CC$ in Section~\ref{S:Calculus}.
  For the subexponential modality, ${!}$, introduction rules, $({!}\to)$ and $(\to{!})$, and permutation rules
  are also the same as in $\CC$, with the natural modification due to a simpler antecedent syntax.
  The contraction
  rule, however, is significantly different, since now it is not controlled by brackets:
$$  \infer[(\CONTR)]{\Delta_1, {!}A, \Delta_2 \to B}
{\Delta_1, {!}A, {!}A, \Delta_2 \to B}
$$

The full set of axioms and rules of $\CCbfp$ is presented in Appendix~II.

This calculus $\CCbfp$ is a conservative fragment of $\MorrillB$, also 
by Morrill and Valent\'{\i}n~\cite{MorVal}. This system could also be used for modelling
medial and parasitic extraction, but is not as fine-grained as the bracketed system, being able
to derive ungrammatical examples like {\sl *``the paper that John sent to''} (see Section~\ref{S:Calculus}).


In order to construct a mapping of $\CC$ into $\CCbfp$,
we define the {\em bracket-forgetting projection (BFP)} of formulae and
meta-formulae that removes all brackets and bracket modalities
($\NMod$ and $\PMod$). The BFP of a formula is again a formula,
but in the language without $\NMod$ and $\PMod$; the BFP of a 
meta-formula is a sequence of formulae. The following lemma is proved by
induction on derivation.

\begin{lemma}\label{Lm:BFP}
If $\CC \vdash \Delta \to C$, then
$\CCbfp \vdash \BFP{\Delta} \to \BFP{C}$.
\end{lemma}


Note that the opposite implication doesn't hold, {\em i.e.,} this mapping is not conservative.
Also, $\CCbfp$ is not a conservative fragment of $\CC$:
in the fragment of $\CC$ without brackets contraction is not admissible.

The second calculus is $\CCwk$, obtained from $\CCbfp$ by adding  weakening  for ${!}$:
$$
\infer[(\WEAK)]{\Delta_1, {!}A, \Delta_2 \to C}{\Delta_1, \Delta_2 \to C}
$$
In $\CCwk$, the ${!}$ connective is equipped with a full set of structural rules (permutation,
contraction, and weakening), {\em i.e.,} it is 
the exponential 
of
linear logic~\cite{Girard}.

The cut rule in $\CCbfp$ and $\CCwk$ can be eliminated by the same ``deep'' strategy
as for $\CC$. On the other hand, since the contraction rule
in these calculi is standard, one can also use the traditional way with mix, like 
in~\cite[Appendix~A]{LMSS92}.

Finally, if we remove ${!}$ with all its rules, we get the Lambek calculus with the unit constant~\cite{LambekUnit}.
We denote it by $\LL$.

%




\section{Undecidability of $\CC$}\label{S:undec}

The main result of this paper is:

\begin{theorem}\label{Th:main}
The derivability problem for $\CC$ is
undecidable.
\end{theorem}

As a by-product of our proof we also obtain undecidability of $\CCbfp$, which was
proved in~\cite{KanKuzSce2} by a different method. We also obtain undecidability
of $\CCwk$,
which also follows from
the results of~\cite{LMSS92}, as shown in~\cite{Kanazawa} and~\cite{deGroote}. 


We prove Theorem~\ref{Th:main} by encoding
derivations in 
generative grammars,
or semi-Thue~\cite{Thue} systems.
A {\em generative grammar} is a quadruple $G = \langle N, \Sigma, P, s \rangle$,
where $N$ and $\Sigma$ are two disjoint alphabets, $s \in N$ is the {\em starting symbol,}
and $P$ is a finite set of {\em productions} (rules) of the form $\alpha \Rightarrow \beta$,
where $\alpha$ and $\beta$ are words over $N \cup \Sigma$. 
The production can be {\em applied} in the following way: $\eta\,\alpha\,\theta \Rightarrow_G
\eta\,\beta\,\theta$, where $\eta$ and $\theta$ are arbitrary (possibly empty) words over
$N \cup \Sigma$. The {\em language generated by $G$} is the set of all
words $\omega$ over $\Sigma$, such that $s \Rightarrow^*_G \omega$, where
$\Rightarrow^*_G$ is the reflexive-transitive closure of
$\Rightarrow_G$.

We use the following classical result by 
Markov~\cite{Markov} and Post~\cite{Post}.

\begin{theorem}\label{Th:Markov}
There exists a generative grammar $G$ that generates an algorithmically undecidable language.
{\rm\cite{Markov}\cite{Post}}
\end{theorem}

In our presentation for every production $(\alpha \Rightarrow \beta) \in P$ we require
$\alpha$ and $\beta$ to be non-empty. This class still includes an
undecidable language (cf.~\cite{Buszko2}). 

Further we use two trivial lemmas about
derivations in a generative grammar:

\begin{lemma}\label{Lm:Gconcat}
If $\alpha_1 \Rightarrow_G^* \beta_1$ and
$\alpha_2 \Rightarrow_G^* \beta_2$, then
$\alpha_1 \alpha_2 \Rightarrow_G^* \beta_1 \beta_2$.
\end{lemma}

\begin{lemma}\label{Lm:Gcut}
If $\alpha \Rightarrow_G^* \beta$ and
$\gamma \Rightarrow_G^* \eta\alpha\theta$, then
$\gamma \Rightarrow_G^* \eta\beta\theta$.
\end{lemma}


The second ingredient we need for our undecidability proof is
the concept of {\em theories} over $\LL$. Let $\Tc$ be a finite
set of sequents in the language of $\LL$. 
Then $\LL + \Tc$ is the calculus 
from $\LL$ by adding sequents from $\Tc$ as
extra axioms. 

In general, the cut rule in $\LL+\Tc$ is not eliminable. However,
the standard cut elimination procedure (see~\cite{Lambek58}) yields
the following {\em cut normalization} lemma:

\begin{lemma}\label{Lm:cutnorm}
If a sequent is derivable in $\LL+\Tc$, then
this sequent has a derivation in which every application of $(\CUT)$ 
has a sequent from $\Tc$ as one of its premises.
\end{lemma}


This lemma yields 
a weak version of the subformula property:

\begin{lemma}\label{Lm:subfmTh}
If $\;\LL + \Tc \vdash \Pi \to A$, and both $\Pi \to A$ and $\Tc$ include
 no occurrences of $\BS$, $\SL$, and $\U$, then there is
a derivation of $\Pi \to A$ in $\LL+\Tc$ that includes no occurrences 
of $\BS$, $\SL$, and $\U$.
\end{lemma}



The third core element of the construction is the $(\SHOOT)$ 
rule which allows to place a specific formula $A$ into an arbitrary place in the sequent.

\begin{lemma}\label{Lm:shoot}
The following rule is admissible in $\CC$:
$$
\infer[(\SHOOT)]
{\Delta_1, {!}\,\NMod A, \Delta_2, \Delta_3 \to C}
{\Delta_1, {!}\,\NMod A, \Delta_2, A, \Delta_3 \to C}
$$
\end{lemma}

\begin{proof}
$$\footnotesize
\infer[(\PERM_2)]{\Delta_1, {!}\,\NMod A, \Delta_2, \Delta_3 \to C}
{\infer[(\CONTRb)]{\Delta_1, \Delta_2, {!}\,\NMod A, \Delta_3 \to C}
{\infer[(\PERM_1)]{\Delta_1, \Delta_2, {!}\,\NMod A, [{!}\,\NMod A], \Delta_3 \to C}
{\infer[({!}\to)]{\Delta_1, {!}\,\NMod A, \Delta_2, [{!}\,\NMod A], \Delta_3 \to C}
{\infer[(\NMod\to)]{\Delta_1, {!}\,\NMod A, \Delta_2, [\NMod A], \Delta_3 \to C}
{\Delta_1, {!}\,\NMod A, \Delta_2, A, \Delta_3 \to C}
}}}}
$$
\end{proof}



Now we are ready to prove Theorem~\ref{Th:main}. Let $G = \langle N, \Sigma, P, s \rangle$ be 
the grammar provided by Theorem~\ref{Th:Markov},
 and the set of variables include $N \cup \Sigma$.
We convert productions of $G$ into Lambek formulae in the following natural way:\\ 
$
\Bc_G = \{ (u_1 \cdot \ldots \cdot u_k) \SL (v_1 \cdot \ldots \cdot v_m) \mid
(u_1 \ldots u_k \Rightarrow v_1 \ldots v_m) \in P \}.
$

For $\Bc_G = \{ B_1, \dots, B_n \}$, we define the following sequences of formulae: 
\begin{align*}
&\Gamma_G = {!}B_1, \dots, {!}B_n, && \Phi_G = {!}(\U \SL ({!}B_1)), \dots, {!}(\U \SL ({!} B_n)),\\
&\widetilde{\Gamma}_G = {!}\,\NMod B_1, \dots, {!}\,\NMod B_n, &&
\widetilde{\Phi}_G = {!}(\U \SL ({!}\, \NMod B_1)), \dots, {!}(\U \SL ({!}\, \NMod B_n)).
\end{align*}
(Since in all calculi we have permutation rules for formulae under
${!}$, the ordering of $\Bc_G$ doesn't matter.)
We also define a theory $\Tc_G$ associated with $G$, as follows:
$
\Tc_G = \{ v_1, \dots, v_m \to u_1 \cdot \ldots \cdot u_k \mid
(u_1 \ldots u_k \Rightarrow v_1 \ldots v_m) \in P \}.
$


\begin{lemma}\label{Lm:main}
The following are equivalent:
\begin{enumerate}
\item $s \Rightarrow^*_G a_1 \dots a_n$ (i.e., 
$a_1 \dots a_n$ belongs to the language defined by  $G$);
\item $\CC \vdash \widetilde{\Phi}_G, \widetilde{\Gamma}_G, a_1, \dots,
a_n \to s$;
\item $\CCbfp \vdash \Phi_G, \Gamma_G, a_1, \dots, a_n \to s$;
\item $\CCwk \vdash \Gamma_G, a_1, \dots, a_n \to s$;
\item $\LL + \Tc_G \vdash a_1, \dots, a_n \to s$.
\end{enumerate}
\end{lemma}


\begin{proof}
\fbox{$1 \Rightarrow 2$}
Proceed by induction on $\Rightarrow^*_G$. 
The base case is handled as follows:
$$\footnotesize
\infer[(\PERM)^*]{\widetilde{\Phi}_G, \widetilde{\Gamma}_G, s \to s}
{\infer[({!}\to)^*]{{!}(\U \SL {!}\,\NMod B_1), {!}\,\NMod B_1, \dots, 
{!}(\U \SL {!}\,\NMod B_n), {!}\,\NMod B_n, s \to s}
{\infer[(\SL\to)^*]{\U \SL {!}\,\NMod B_1, {!}\,\NMod B_1, \dots, 
\U \SL {!}\,\NMod B_n, {!}\,\NMod B_n, s \to s}
{{!}\,\NMod B_1 \to {!}\,\NMod B_1 & \ldots
& {!}\,\NMod B_n \to {!}\,\NMod B_n & 
\infer[(\U\to)^*]{\U, \dots, \U, s \to s}{s \to s}
}}}
$$

For the induction step let the last production be $u_1 \dots u_k \Rightarrow v_1 \dots v_m$,
{\em i.e.,} 
$
s \Rightarrow_G^* \eta\, u_1 \dots u_k \, \theta \Rightarrow_G \eta\, v_1 \dots v_m\, \theta.
$

Then, since ${!}\,\NMod ((u_1 \cdot \ldots \cdot u_k) \SL (v_1 \cdot \ldots \cdot v_m))$ 
is in $\widetilde{\Gamma}_G$, we enjoy the following: 
$$\footnotesize
\infer[(\SHOOT)]
{\widetilde{\Phi}_G, \widetilde{\Gamma}_G, \eta, v_1, \dots, v_m, \theta \to s}
{\infer[(\SL\to)]
{\widetilde{\Phi}_G, \widetilde{\Gamma}_G, \eta, (u_1 \cdot\ldots\cdot u_k) \SL (v_1 \cdot\ldots\cdot v_m), v_1, \dots, v_m, \theta \to s}
{\infer[(\to\cdot)^*]{\strut v_1, \dots, v_m \to v_1 \cdot\ldots\cdot v_m}{v_1 \to v_1 & \dots & v_m \to v_m} & 
\infer[(\cdot\to)^*]{\widetilde{\Phi}_G, \widetilde{\Gamma}_G, \eta, u_1 \cdot\ldots\cdot u_k, \theta \to s}
{\widetilde{\Phi}_G, \widetilde{\Gamma}_G, \eta, u_1, \dots, u_k, \theta \to s}}
}
$$
Here  $\widetilde{\Phi}_G, \widetilde{\Gamma}_G, \eta, u_1, \dots, u_k, \theta \to s$ is derivable in $\CC$ by
induction hypothesis, and the $(\SHOOT)$ rule is admissible due to Lemma~\ref{Lm:shoot}.



\fbox{$2 \Rightarrow 3$}
Immediately by Lemma~\ref{Lm:BFP}, since $\Phi_G = \BFP{\widetilde{\Phi}_G}$ and
$\Gamma_G = \BFP{\widetilde{\Gamma}_G}$.

\fbox{$3 \Rightarrow 4$}
For each formula ${!}(\U \SL {!} B_i)$ from $\Phi_G$ the sequent $\Lambda \to 
{!}(\U\SL {!}B_i)$ is derivable in $\CCwk$ by consequent application of
$(\WEAK)$, $(\to\SL)$, and $(\to{!})$ to the $\Lambda \to \U$ axiom.
%
The sequent $\Phi_G, \Gamma_G, a_1, \dots, a_n \to s$  is derivable
in $\CCbfp$ and therefore in $\CCwk$, and applying $(\CUT)$ for each formula of $\Phi_G$ 
yields $\Gamma_G, a_1, \dots, a_n \to s$.

\fbox{$4 \Rightarrow 5$}
In this part of our proof we follow~\cite{LMSS92} and~\cite{Kanazawa}. 
 Consider the derivation
of $\Gamma_G, a_1, \dots, a_n \to s$ in $\CCwk$ (recall that by default
all derivations are cut-free) and remove all the formulae of the form ${!}B$ from
all sequents in this derivation. After this transformation the rules not
operating with ${!}$ remain valid. Applications of
$(\PERM_i)$, $(\WEAK)$, and $(\CONTR)$ do not alter the sequent. The
$(\to{!})$ rule is never applied in the original derivation, since our sequents never have formulae of
the form ${!}B$ in their succedents. Finally, an application of $({!}\to)$,
$$\footnotesize
\infer[,]
{\Delta_1, \Delta_2 \to C}
{\Delta_1, (u_1 \cdot \ldots \cdot u_k) \SL (v_1 \cdot \ldots \cdot v_m), \Delta_2 \to C}
$$
is simulated in $\LL + \Tc_G$ in the following way:
$$
\scriptsize
\infer[(\CUT)]
{\Delta_1, \Delta_2 \to C}
{\infer[(\to\SL)]{\Lambda \to (u_1 \cdot \ldots \cdot u_k) \SL (v_1 \cdot \ldots \cdot v_m)}
{\infer[(\cdot\to)^*]{\strut v_1 \cdot \ldots \cdot v_m \to  u_1 \cdot \ldots \cdot u_k}
{v_1, \dots, v_m \to u_1 \cdot \ldots \cdot u_k}} &
 \Delta_1, (u_1 \cdot \ldots \cdot u_k) \SL (v_1 \cdot \ldots \cdot v_m), \Delta_2 \to C}
$$


\fbox{$5 \Rightarrow 1$}
In this part we follow~\cite{LMSS92}.
Let $\LL + \Tc_G \vdash a_1, \dots, a_n \to s$. By Lemma~\ref{Lm:subfmTh}, this sequent has
a derivation without occurrences of $\BS$, $\SL$, and $\U$. In other words, all formulae
in this derivation are built from variables using only the product.
Since it is associative, we can omit parenthesis in the formulae;
we shall also omit the ``$\cdot$''s. The rules used in this derivation can
now  be written as follows:
$$
\infer[(\to\cdot)]{\beta_1 \beta_2 \to \alpha_1 \alpha_2}{\beta_1 \to \alpha_1 & \beta_2 \to \alpha_2}
\qquad
\infer[(\CUT)]{\eta\beta\theta \to \gamma}{\beta \to \alpha & \eta \alpha \theta \to \gamma}
$$
The $(\cdot\to)$ rule is trivial. The axioms are productions of $G$ with
the arrows inversed, and $\alpha \to \alpha$. By induction, using Lemmas~\ref{Lm:Gconcat} and~\ref{Lm:Gcut}, we show that
if $\beta \to \alpha$ is derivable using these rules and axioms, then $\alpha \Rightarrow_G^* \beta$. Now the
derivability of $a_1, \dots, a_n \to s$
implies $s \Rightarrow_G^* a_1 \ldots a_n$.
%
\end{proof}

Lemma~\ref{Lm:main} and Theorem~\ref{Th:Markov} conclude the proof of 
Theorem~\ref{Th:main}.

\section{A Decidable Fragment}\label{S:effective}


The undecidability results from the previous section are somewhat
unfortunate, since the new operations added to $\LL$ have good
linguistic motivations~\cite{MorVal}\cite{MorrillBook}. As a compensation,
in this section we show NP-decidability for a substantial fragment of
$\CC$, introduced by Morrill and Valent\'{\i}n~\cite{MorVal}
(see Definition~\ref{Df:BNNC} below).
This complexity upper bound is tight, since the original Lambek calculus 
is already known to be NP-complete~\cite{PentusNP}.
Notice that Morrill and Valent\'{\i}n present an exponential time
algorithm for deciding derivability in this fragment; this algorithm
was implemented as part of a parser called CatLog~\cite{CatLog}. 


First we recall the standard notion of {\em polarity} of occurrences of subformulae in
a formula. Every formula occurs positively in itself; subformula polarities get inverted (positive becomes
negative and vice versa) when descending into denominators of $\BS$ and $\SL$ and also for the left-hand side of
the sequent; brackets and all unary operations don't change polarity.
All inference rules of $\CC$ respect polarity: a positive (resp., negative) occurrence of a subformula in the
premise(s) of the rule translates into a positive (resp., negative) occurrence in the goal.

\begin{definition}\label{Df:BNNC}
An $\CC$-sequent $\Gamma \to B$ obeys the bracket non-negative condition, if any
negative occurrence of a subformula of the form ${!}A$ in $\Gamma \to B$ includes
neither a positive occurrence of a subformula of the form $\NMod C$, nor a negative
occurrence of a subformula of the form $\PMod C$.
\end{definition}

Note that sequents used in our undecidability proof are exactly the minimal violations
of this bracket non-negative condition. 

\begin{theorem}\label{Th:NP}
The derivability problem in $\CC$ for sequents that obey the bracket non-negative condition
belongs to the {\rm NP} class.
\end{theorem}



Derivations in $\CC$ are a bit incovenient for complexity estimations, since redundant applications
of permutation rules could make the proof arbitrarily large without increasing its ``real'' complexity.
In order to get rid of that, we introduce a {\em generalised form of permutation rule:}
$$
\infer[(\PERM)^*]{\Delta'_0, {!}A_{i_1}, \Delta'_1, {!}A_{i_2}, \Delta'_2, \ldots, \Delta'_{i_{k-1}}, {!}A_{i_k}, \Delta'_{i_k} \to C}
{\Delta_0, {!}A_1, \Delta_1, {!}A_2, \Delta_2, \ldots, \Delta_{k-1}, {!}A_k, \Delta_k \to C}
$$
where the sequence $\Delta'_0, \ldots, \Delta'_k$ coincides with $\Delta_0, \dots, \Delta_k$, and
$\{ i_1, \dots, i_k \} = \{ 1, \dots, k \}$. Obviously, $(\PERM)^*$ is admissible in $\CC$, and it subsumes
$(\PERM_{1,2})$, so further
we consider a formulation of $\CC$ with $(\PERM)^*$ instead of $(\PERM_{1,2})$. Several consecutive applications
of $(\PERM)^*$ can be merged into one. We call a derivation {\em normal,}  if it doesn't
contain consecutive applications of $(\PERM)^*$. If a sequent is derivable in $\CC$, then it has a normal
cut-free derivation.

\begin{lemma}\label{Lm:NP}
Every normal cut-free derivation of a sequent that obeys bracket non-negative
restriction is of quadratic size (number of rule applications) w.r.t.\ the size of the
goal sequent.
\end{lemma}


\begin{proof}
Let us call $(\CONTRb)$ and $(\PERM)^*$ {\em structural}
rules, and all others {\em logical}.

First, we track all pairs of brackets that occur in this derivation.
Pairs of brackets are in one-to-one correspondence with applications of $(\NMod\to)$ or $(\to\PMod)$ rules that introduce them.
Then a pair of brackets either traces down to the goal sequent, or gets destroyed by an application of $(\PMod\to)$, $(\to\NMod)$,
or $(\CONTRb)$. Therefore, the total number of $(\CONTRb)$ applications is less or equal to the number of $(\NMod\to)$ and
$(\to\PMod)$ applications. Each $(\NMod\to)$ application introduces a negative occurrence of a $\NMod C$ formula; each $(\to\PMod)$ occurrence
introduces a positive occurrence of a $\PMod C$ formula. Due to the bracket non-negative condition these formulae are never contracted
({\em i.e.,} could not occur in a ${!}A$ to which $(\CONTRb)$ is applied),
and therefore they trace down to {\em distinct} subformula occurrences in the goal sequent. Hence, the total number of $(\CONTRb)$ applications is bounded
by the number of subformulae of a special kind in the goal sequent, in other words, it is bounded by the size of the sequent.

Second, we bound the number of logical rules applications. Each logical rule introduces exactly one connective occurrence. Such an occurrence
traces down either to a connective occurrence in the goal sequent, or to an application of $(\CONTRb)$ that merges this occurrence
with the corresponding occurrence in the other ${!}A$. If $n$ is the size of the goal sequent, then the first kind of occurrences
is bounded by $n$; for the second kind, notice that each application of $(\CONTRb)$ merges not more than $n$ occurrences (since
the size of the formula being contracted, ${!}A$, is bounded by $n$ due to the subformula property), and the total number of $(\CONTRb)$ applications is also
bounded by $n$. Thus, we get a quadratic bound for the number of logical rule applications.

Third, the derivation is a tree with binary branching, so the number of leafs (axioms instances) in this tree
is equal to the number of branching points plus one. Each branching point is an application of a logical rule
(namely, $(\BS\to)$, \mbox{$(\SL\to)$,} or $(\to\cdot)$). Hence, the number of axiom instances is  bounded quadratically.

Finally, the number of $(\PERM)^*$ applications is also quadratically bounded, since each application
of $(\PERM)^*$ in a normal proof is preceded by an application of another rule or by an axiom instance.
\end{proof}

\begin{proof}[of Theorem~\ref{Th:NP}]
The normal derivation of a sequent obeying the bracket non-negative condition is
an NP-witness for derivability: it is of polynomial size, and  correctness
is checked in linear time (w.r.t.\ the size of the derivation).
\end{proof}

For the case without brackets, $\CCbfp$, considered in our earlier paper~\cite{KanKuzSce2}, 
the NP-decidable fragment is substantially smaller. Namely, it includes
only sequents in which ${!}$ can be applied only to variables. Indeed, as soon as
we allow formulae of implication nesting depth at least 2 under ${!}$, the derivablity
problem for $\CCbfp$ becomes undecidable~\cite{KanKuzSce2}. In contrast to $\CCbfp$,
in $\CC$, due to the non-standard contraction rule,
 brackets control the number of $(\CONTRb)$ applications in the proof, and this allows
to construct an effective decision algorithm for derivability of a broad class of sequents, where,
for example, any formulae without bracket modalities can be used under ${!}$.  
Essentially, the only problematic situation, that gives rise to undecidability
(Theorem~\ref{Th:main}), is the construction where one forcedly removes the brackets that
appear in the $(\CONTRb)$ rule, {\em i.e.,} uses constructions like ${!} \NMod B$ (as in
our undecidability proof). The idea of the bracket non-negative condition is to rule out
such situations while keeping all other constructions allowed, as they don't violate
decidability~\cite{MorVal}. 



\section{Conclusions and Future Work}\label{S:future}
In this paper we study an extension of the Lambek calculus with subexponential and bracket modalities.
Bracket modalities were introduced by Morrill~\cite{Morrill1992} and Moortgat~\cite{Moortgat1996} in order to represent
the linguistic phenomenon of islands~\cite{Ross}. The interaction of subexponential and bracket
modalities was recently studied by Morrill and Valent\'{\i}n~\cite{MorVal} in order to represent
correctly the phenomenon of medial and parasitic extraction~\cite{Ross}\cite{Barry}. We prove that the calculus of
Morrill and Valent\'{\i}n is undecidable, thus solving a problem left open in~\cite{MorVal}. Morrill and Valent\'{\i}n also considered the so-called
bracket non-negative fragment of this calculus, for which they presented an exponential time
derivability decision procedure. We improve their result by showing that this problem is in NP.

Our undecidability proof is based on encoding semi-Thue systems by means of sequents that lie
just outside the bracket non-negative fragment. More precisely,
the formulae used in our encoding
are of the from ${!}\,\NMod A$, where $A$ is a pure Lambek formula of order 2. It remains for
further investigation whether these formulae could be simplified.


Our undecidability proof could be potentially made stronger by restricting
the language. Now we use three connectives of the original Lambek calculus:
\mbox{$\SL$, $\cdot$, and $\U$,} plus $\NMod$ and ${!}$.
One could get rid of $\U$ by means of the substitution from~\cite{KuznUnit}.
Going further, one might also encode a more clever construction
by Buszkowski~\cite{Buszko} in order to restrict ourselves further
to the product-free one-division fragment. Finally, one could adopt substitutions
from~\cite{KanovichNeutrals} and obtain undecidability for the language with
only one variable.

There are also several other linguistically motivated extensions of the Lambek calculus
(see, for instance,~\cite{MorrillBook}\cite{MootRetore}\cite{MorrillPhilosophy}) 
and their algorithmic and logical properties should be investigated.

 


\newpage

\section*{Appendix~I. Cut Elimination Proof for $\CC$}

In this section we give a complete proof of Lemma~\ref{Lm:cutelim}, which is the main
step of cut elimination in $\CC$ (Theorem~\ref{Th:cutelim}).

\begin{proof}
Proceed by nested induction on two parameters:

\begin{enumerate}
\item The complexity $\kappa$ of the formula $A$ being cut.
\item The total number $\sigma$ of rule applications in $\Der_{\mathrm{left}}$ and $\Der_{\mathrm{right}}$. %
\end{enumerate}

In each case either $\kappa$ gets reduced, or $\kappa$ remains the
same 
and $\sigma$ gets reduced.

{\bf Case 1 (axiomatic).} One of the premises of $(\CUT)$ is an axiom
of the form $A \to A$. Then the other premise coincides with the goal, and
cut disappears.

{\bf Case 2 (left non-principal).} 

{\it Subcase 2.a.} The last rule in $\Der_{\mathrm{left}}$ is one of the one-premise rules operating only on
the left-hand side of the sequent: $(\cdot\to)$, $(\U\to)$, $(\NMod\to)$, $(\PMod\to)$, $({!}\to)$,
$(\PERM_i)$, $(\CONTR)$. Denote this rule by $(R)$. Notice that $(R)$ can be applied in any context,
and transform the derivation in the following way:
$$
\infer[(\CUT)]{\Delta (\Pi) \to C}
{\infer[(R)]{\Pi \to A}{\Pi' \to A} & \Delta (A) \to C}
\qquad
\text{\raisebox{1em}{$\leadsto$}}
\qquad
\infer[(R)]{\Delta (\Pi) \to C}
{\infer[(\CUT)]{\Delta (\Pi') \to C}{\Pi' \to A & \Delta (A) \to C}}
$$

The $\sigma$ parameter gets reduced, therefore the new cut is eliminable by induction hypothesis.

{\it Subcase 2.b.} The last rule in $\Der_{\mathrm{left}}$ is 
$(\BS\to)$ or $(\SL\to)$. Then the derivation fragment

$$
\infer[(\CUT)]{\Delta (\Pi (\Pi', E \BS F)) \to C}
{\infer[(\BS\to)]{\Pi (\Pi', E \BS F) \to A}{\Pi' \to E & \Pi(F) \to A} & \Delta(A) \to C}
$$
is transformed into 
$$
\infer[(\BS\to)]{\Delta (\Pi (\Pi', E \BS F)) \to C}
{\Pi' \to E & \infer[(\CUT)]{\Delta (\Pi (F)) \to C}{\Pi(F) \rangle \to A & \Delta(A) \to C}}
$$

Again, $\sigma$ decreases. The $(\SL\to)$ case is handled symmetrically.


{\bf Case 3 (deep).}
The last rule applied on the left is $(\to{!})$. Then the cut rule application has the following form:
$$
\infer[(\CUT)]{\Delta \langle {!}\Pi \rangle \to C}{
\infer[(\to{!})]{{!}\Pi \to {!}A}{{!}\Pi \to A} & \Delta \langle {!}A \rangle \to C}
$$
The right premise, $\Delta \langle {!}A \rangle \to C$, has a cut-free derivation tree $\Der_{\mathrm{right}}$.
Let us trace the designated occurrence of ${!}A$ in $\Der_{\mathrm{right}}$. The trace can branch if
$(\CONTR)$ is applied to this formula. Each branch of the trace ends either with an axiom (${!}A \to {!}A$) leaf or
with an application of $({!}\to)$ that introduces ${!}A$.

The axiom ${!}A \to {!}A$ can be reduced to $A \to A$ by consequent application of
$({!}\to)$ and $(\to{!})$. 
%
%
Therefore, without loss of generality, we can assume that all branches lead to applications of
$({!}\to)$. The whole picture is shown on Figure~\ref{Fig:deepcut1}.
\begin{figure}
\includegraphics[scale=.7]{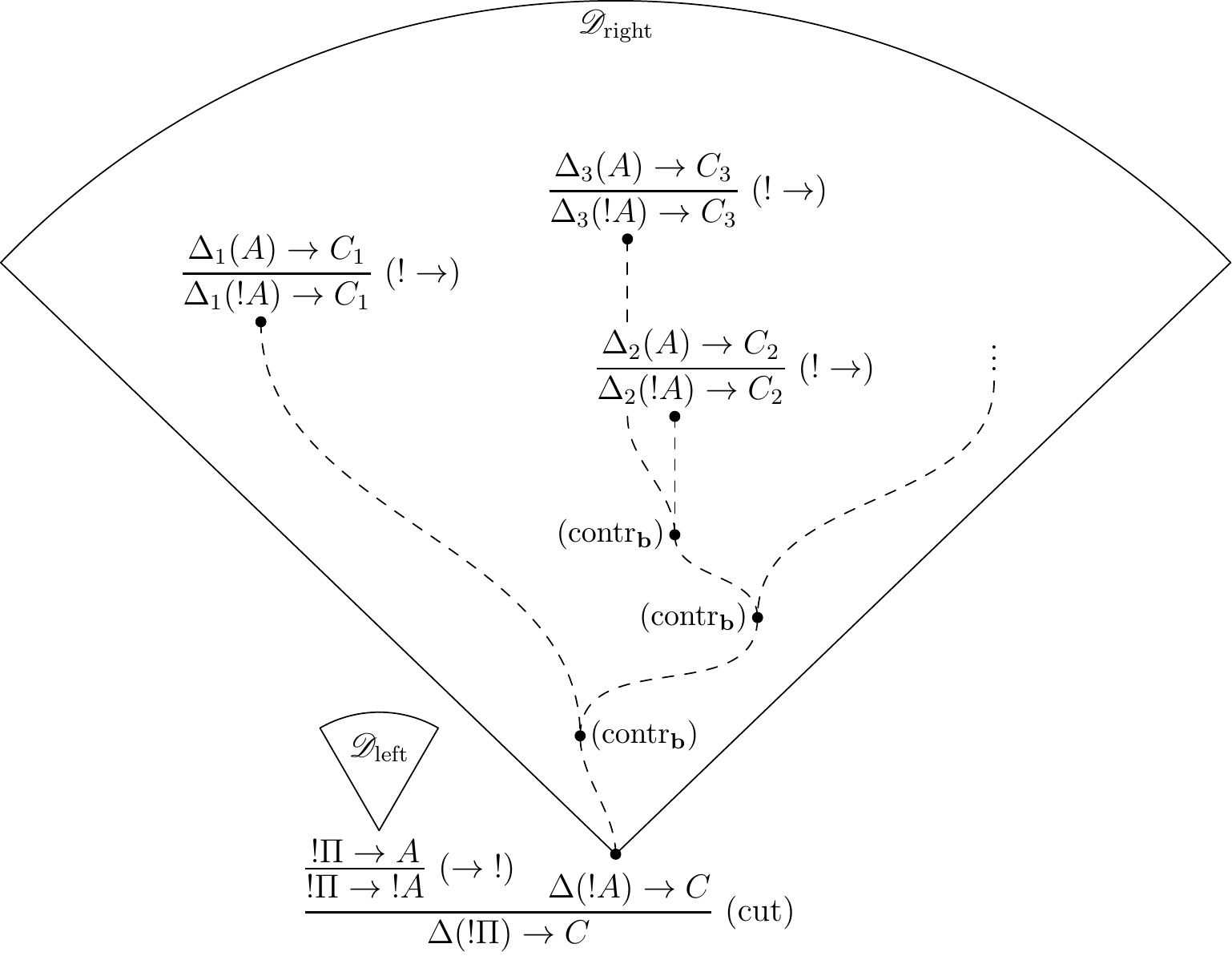}
\caption{} \label{Fig:deepcut1}
\end{figure}

In $\Der_{\mathrm{right}}$ we replace the designated occurrences of ${!}A$ with ${!}\Pi$ along the traces.
The applications of $(\CONTR)$ remain valid; if there were permutation rules applied,
we replace such a rule with a series of permutations for each formula in ${!}\Pi$. Other
rules do not operate ${!}A$ and therefore remain intact. After this replacement
applications of $({!}\to)$ tranform into applications of $(\CUT)$ with $\Pi \to A$ as the
left premise (Figure~\ref{Fig:deepcut2}).
One case could go through several instances
of $({!}\to)$ with the active ${!}A$, like $\Delta_2$ and $\Delta_3$ in the example; in this case
we go from top to bottom.
\begin{figure}
\includegraphics[scale=.7]{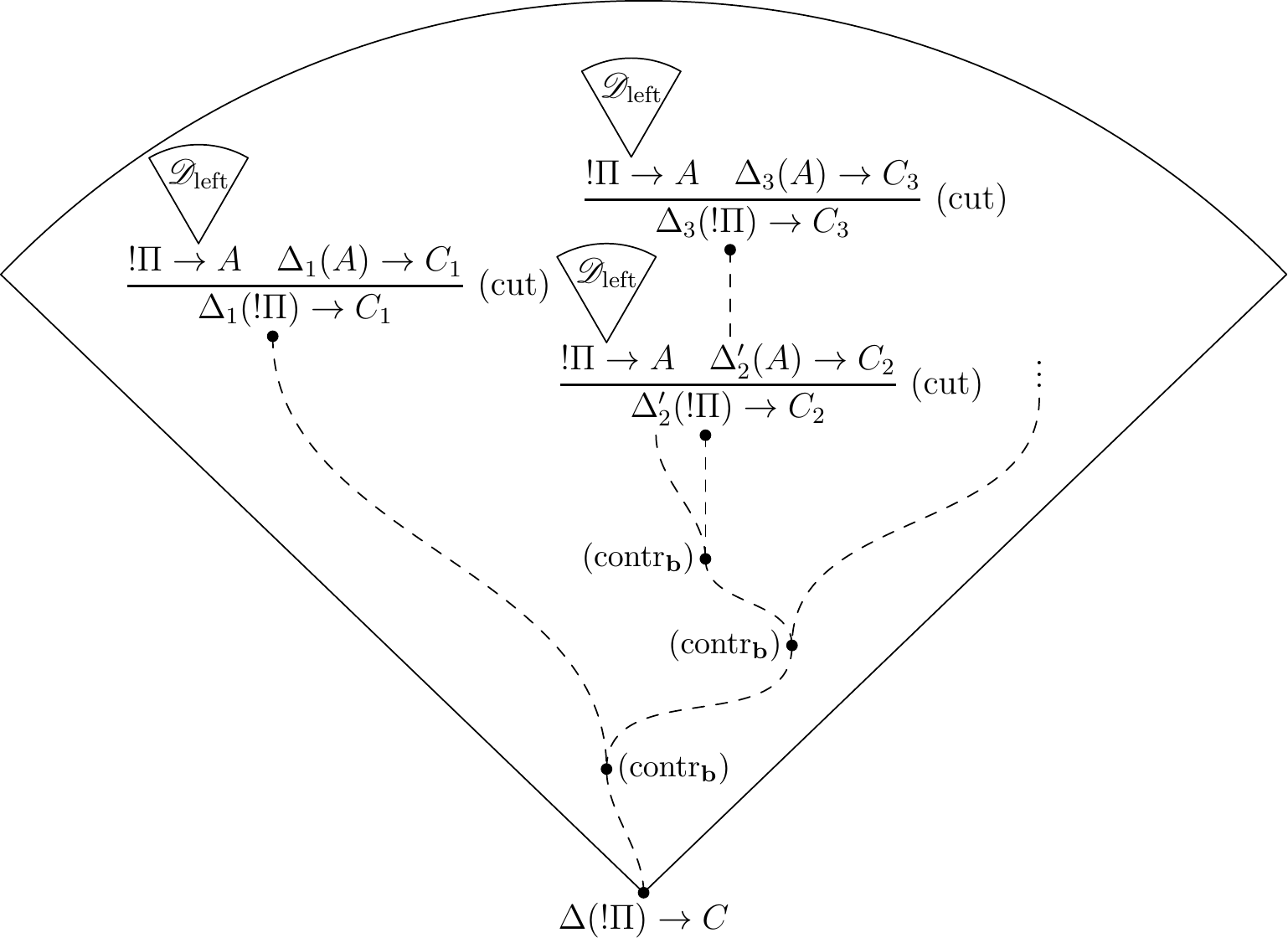}
\caption{} \label{Fig:deepcut2}
\end{figure}


The new cuts have lower $\kappa$ (the cut formula is $A$ instead of ${!}A$), and therefore they
are eliminable by induction hypothesis.

{\bf Case 4 (principal).} In the so-called principal case, the last rules
both in $\Der_{\mathrm{left}}$ and in $\Der_{\mathrm{right}}$ introduce the
main connective of the formula $A$ being cut. Note that $A$ here is not
of the form ${!}A'$ (this is the previous case). In the principal case,
the $\kappa$ parameter gets reduced, and therefore the induction hypothesis
can be applied to eliminate the new cut(s) that arise after the transformation.

{\em Subcase 4.a: $(\to\BS)$ vs. $(\BS\to)$ or $(\to\SL)$ vs. $(\SL\to)$.}
In this case $A = A_1 \BS A_2$ (the $\SL$ case is handled symmetrically), and
the derivation fragment 
$$
\infer[(\CUT)]{\Delta (\Phi, \Pi) \to C}
{\infer[(\to\BS)]{\Pi \to A_1 \BS A_2}{A_1, \Pi \to A_2} & 
\infer[(\BS\to)]{\Delta (\Phi, A_1 \BS A_2) \to C}
{\Phi \to A_1 & \Delta (A_2) \to C}}
$$
transforms into
$$
\infer[(\CUT)]{\Delta (\Phi, \Pi) \to C}
{\Phi \to A_1 & \infer[(\CUT)]{\Delta (A_1, \Pi) \to C}
{A_1, \Pi \to A_2 & \Delta (A_2) \to C}}
$$


{\em Subcase 4.b. $(\to\cdot)$ vs. $(\cdot\to)$.}
In this case $A = A_1 \cdot A_2$, and the derivation fragment
$$
\infer[(\CUT)]{\Delta (\Pi_1, \Pi_2) \to C}
{\infer[(\to\cdot)]{\Pi_1, \Pi_2 \to A_1 \cdot A_2}{\Pi_1 \to A_1 & \Pi_2 \to A_2} &
\infer[(\cdot\to)]{\Delta (A_1 \cdot A_2) \to C}
{\Delta (A_1, A_2) \to C}}
$$
transforms into
$$
\infer[(\CUT)]{\Delta (\Pi_1, \Pi_2) \to C}
{\Pi_2 \to A_2 & \infer[(\CUT)]{\Delta (\Pi_1, A_2) \to C}
{\Pi_1 \to A_1 & \Delta (A_1, A_2) \to C}}
$$

{\em Subcase 4.c. $(\to\U)$ vs. $(\U\to)$.}
In this case $A = \U$:
$$
\infer[(\CUT)]{\Delta (\Lambda) \to C}
{\infer[(\to\U)]{\Lambda \to \U}{} & \infer[(\U\to)]{\Delta (\U) \to C}
{\Delta (\Lambda) \to C}}
$$
The cut disappears, since its goal coincides with the premise of $(\U\to)$.

{\em Subcase 4.d. $(\to\NMod)$ vs. $(\NMod\to)$.}
In this case $A = \NMod A'$, and the derivation fragment
$$
\infer[(\CUT)]{\Delta ([\Pi]) \to C}
{\infer[(\to\NMod)]{\Pi \to \NMod A'}{[\Pi] \to A'} &
\infer[(\NMod\to)]{\Delta ([\NMod A']) \to C}
{\Delta (A') \to C}}
$$
transforms into
$$
\infer[(\CUT)]{\Delta ([\Pi]) \to C}
{[\Pi] \to A' & \Delta (A') \to C}
$$

{\em Subcase 4.e. $(\to\PMod)$ vs. $(\PMod\to)$.}
In this case $A = \PMod A'$, and the derivation fragment
$$
\infer[(\CUT)]{\Delta ([\Pi]) \to C}
{\infer[(\to\PMod)]{[\Pi] \to \PMod A'}{\Pi \to A'} &
\infer[(\PMod\to)]{\Delta (\PMod A' \to C}
{\Delta ([A']) \to C}}
$$
transforms into
$$
\infer[(\CUT)]{\Delta ([\Pi]) \to C}
{\Pi \to A' & \Delta ([A']) \to C}
$$

{\bf Case 5 (right non-principal).} In the remaining cases, $A$ is not of the form
${!}A'$ (therefore the last rule of $\Der_{\mathrm{right}}$ is not $(\to{!})$;
it is also not $(\to\U)$, since there is nothing to cut in an empty antecedent)
and the last rule of $\Der_{\mathrm{right}}$ does not operate on $A$. 
In this case, the cut gets propagated upwards to $\Der_{\mathrm{right}}$, decreasing
$\sigma$ with the same $\kappa$.
\end{proof}


\section*{Appendix II. Axioms and Rules of $\CCbfp$}

$$
\infer{A \to A}{}
\qquad
\infer{\Lambda \to \U}{}
$$
$$
\infer[(\SL\to)]{\Delta_1, C \SL B, \Gamma, \Delta_2 \to D}
{\Gamma \to B & \Delta_1, C, \Delta_2 \to D}
\quad
\infer[(\to\SL)]{\Gamma \to C \SL B}{\Gamma, B \to C}
\quad
\infer[(\cdot\to)]{\Delta_1, A \cdot B, \Delta_2 \to D}
{\Delta_1, A, B, \Delta_2 \to D}
$$
$$
\infer[(\BS\to)]{\Delta_1, \Gamma, A \BS C, \Delta_2 \to D}
{\Gamma \to A & \Delta_1, C, \Delta_2 \to D}
\quad
\infer[(\to\BS)]{\Gamma \to A \BS C}{A, \Gamma \to C}
\quad
\infer[(\to\cdot)]
{\Gamma_1, \Gamma_2 \to A \cdot B}
{\Gamma_1 \to A & \Gamma_2 \to B}
$$
$$
\infer[({!}\to)]{\Gamma_1, {!}A, \Gamma_2 \to B}{\Gamma_1, A, \Gamma_2 \to B}
\quad
\infer[(\PERM_1)]{\Delta_1, \Gamma, {!}A, \Delta_2 \to B}
{\Delta_1, {!}A, \Gamma, \Delta_2 \to B}
\quad
\infer[(\PERM_2)]{\Delta_1, {!}A, \Gamma, \Delta_2 \to B}
{\Delta_1, \Gamma, {!}A, \Delta_2 \to B}
$$
$$
\infer[(\U\to)]{\Delta_1, \U, \Delta_2 \to A}{\Delta_1, \Delta_2 \to A}
\quad
\infer[(\to{!})]{{!}A_1, \dots, {!}A_n \to {!}A}
{{!}A_1, \dots, {!}A_n \to A}
\quad
\infer[(\CONTR)]{\Delta_1, {!}A, \Delta_2 \to B}
{\Delta_1, {!}A, {!}A, \Delta_2 \to B}
$$
$$
\infer[(\CUT)]{\Delta_1, \Pi, \Delta_2 \to C}{\Pi \to A & \Delta_1, A, \Delta_2 \to C}
$$

\newpage

\begin{thebibliography}{XX}
\bibitem{Abrusci} V. M. Abrusci. A comparison between Lambek syntactic calculus and intuitionistic linear
propositional logic. Zeitschr. f\"ur math. Log. Grundl. Math. (Math. Logic Quart.), 36:11--15, 1990.
\bibitem{Ajdukiewicz} K. Ajdukiewicz. Die syntaktische Konnexit\"at. Studia Philosophica, Vol.~1, 1--27, 1935.
\bibitem{BarHillel} Y. Bar-Hillel. A quasi-arithmetical notation for syntactic description. Language, Vol.~29, 47--58, 1953.
\bibitem{Barry} G. Barry, M. Hepple, N. Leslie, G. Morrill. Proof figures and structural operators for categorial grammar.
Proc. 5th Conference of the European Chapter of ACL, Berlin, 1991.
\bibitem{BraunerPaivaBRICS} T. Bra\"uner, V. de Paiva. Cut elimination for full intuitionstic linear logic.
BRICS Report RS-96-10, April 1996.
\bibitem{BraunerPaivaCSL} T. Bra\"uner, V. de Paiva. A formulation of linear logic based on dependency
relations. Proc. CSL 1997, LNCS vol.~1414, Springer, 1998, 129--148.
\bibitem{BraunerS5} T. Bra\"uner. A cut-free Gentzen formulation of modal logic S5. Log. J. IGPL, 8(5):629--643, 2000.
\bibitem{Buszko}  W. Buszkowski. Some decision problems in the theory of
syntactic categories. Zeitschr. f\"{u}r math. Logik und
 Grundl. der Math. (Math. Logic Quart.), Vol. 28, 
  539--548, 1982.
\bibitem{BuszkoTLiG} W. Buszkowski. Type logics in grammar. In: Trends in Logic: 50 Years of Studia Logica, Springer, 2003,
337--382.
\bibitem{Buszko2} W. Buszkowski. Lambek calculus with nonlogical axioms.
Language and Grammar. CSLI Lect. Notes vol.~168, 2005, 77--93.
\bibitem{Carpenter} B. Carpenter. Type-logical semantics. MIT Press, 1997.
\bibitem{DikovskyDekhtyar} M. Dekhtyar, A. Dikovsky. Generalized categorial dependency
grammars. Trakhtenbrot/Festschrift, LNCS vol.~4800, Springer, 2008, 230--255.
\bibitem{Eades} H. Eades III, V. de Paiva. Multiple conclusion linear logic: cut elimination and more.
Proc. LFCS 2016. LNCS vol.~9537, 2015, 90--105.
\bibitem{Gentzen35} G. Gentzen. Untersuchungen \"uber das logische Schlie{\ss}en I.
Mathematische Zeitschrift, Vol.~39, 176--210, 1935.
\bibitem{Girard} J.-Y. Girard. Linear logic. Theor. Comput. Sci. 50:1--102, 1987.
\bibitem{deGroote} Ph. de Groote. On the expressive power of the Lambek calculus
extended with a structural modality. Language and Grammar. CSLI Lect. Notes,
vol.~168, 2005, 95--111.
\bibitem{Kanazawa} M. Kanazawa. Lambek calculus: Recognizing power and complexity.
In: J. Gerbrandy et al. (eds.). JFAK. Essays
dedicated to Johan van Benthem on the occasion of his 50th birthday.
Vossiuspers, Amsterdam Univ. Press,
1999.
\bibitem{KanovichNeutrals} M. Kanovich. The complexity of neutrals in linear
logic. Proc. LICS '95, 1995, 486--495.
\bibitem{KanKuzSce} M. Kanovich, S. Kuznetsov, A. Scedrov.
On Lambek's restriction in the presence of exponential modalities.
Proc. LFCS '16. LNCS vol.~9537, 2015, 146--158.
\bibitem{KanKuzSce2} M. Kanovich, S. Kuznetsov, A. Scedrov.
Undecidability of the Lambek calculus with a relevant
modality. Proc. FG '15 and FG '16. LNCS vol.~9804, 2016, 240--256
(arXiv: 1601.06303).
\bibitem{KanKuzMorSce} M. Kanovich, S. Kuznetsov, G. Morrill, A. Scedrov.
A polynomial time algorithm for the Lambek calculus with brackets of bounded order.
arXiv preprint 1705.00694, 2017. Submitted for publication.
\bibitem{KuznUnit} S. L. Kuznetsov. On the Lambek calculus with a unit and one division. Moscow Univ. Math. Bull., 66:4 (2011), 173--175.
\bibitem{Lambek58}
J. Lambek. The mathematics of sentence structure. Amer. Math.
Monthly, Vol. 65, No. 3, 154--170, 1958.
\bibitem{LambekUnit}
J. Lambek. Deductive systems and categories II: Standard constructions
and closed categories. Category Theory, Homology Theory and their Applications
I. Lect. Notes Math. vol.~86, Springer, 1969, 76--122.
\bibitem{LMSS92} P. Lincoln, J. Mitchell, A. Scedrov,
N. Shankar. Decision problems for propositional linear logic.
APAL, 56:239--311, 1992.
\bibitem{Markov}
A.~Markov. On the impossibility of certain algorithms in the theory of
associative systems. Doklady Acad. Sci. USSR (N. S.), 55 (1947), 583--586.
\bibitem{Moortgat1996} M. Moortgat. Multimodal linguistic inference. J. Log. Lang. Inform.,
5(3,4):349--385, 1996.
\bibitem{MootRetore} R. Moot, C. Retor\'{e}. The logic of categorial grammars: a deductive
account of natural language syntax and semantics. Springer, 2012.
\bibitem{Morrill1992} G. Morrill. Categorial formalisation of relativisation: pied piping,
islands, and extraction sites. Technical Report LSI-92-23-R, Universitat Polit\`{e}cnica de
Catalunya, 1992.
\bibitem{MorrillBook} G. V. Morrill. Categorial grammar: logical syntax, semantics, and processing.
Oxford University Press, 2011.
\bibitem{CatLog} G. Morrill. CatLog: a categorial parser/theorem-prover. System demonstration,
LACL 2012, Nantes, 2012.
\bibitem{MorrillPhilosophy} G. Morrill. Grammar logicised: relativisation. Linguistics and Philosophy, 40(2): 119--163, 2017.
\bibitem{MorVal} G. Morrill, O. Valent\'\i n. Computational coverage of TLG: Nonlinearity.
Proc. NLCS '15 (EPiC Series, vol.~32), 2015. P.~51--63.
\bibitem{NigamMiller} V. Nigam, D. Miller. Algorithmic specifications in linear logic with subexponentials.
Proc. PPDP '09, ACM, 2009. P.~129--140.
\bibitem{PentusNP} M. Pentus. Lambek calculus is NP-complete. Theor. Comput. Sci. 357(1):186--201, 2006.
\bibitem{Pentus2010} M. Pentus. A polynomial time algorithm for Lambek grammars of bounded order. Linguistic
Analysis, 36(1--4):441--471, 2010.
\bibitem{Post} E. L. Post. Recursive unsolvability of a problem of Thue.
J. Symb. Log., 12 (1947), 1--11.
\bibitem{Ross} J. R. Ross. Constraints on variables in syntax. Ph. D. Thesis, MIT, 1967.
\bibitem{Steedman} M. Steedman. The syntactic process. MIT Press, 2000.
\bibitem{Thue} A. Thue. Probleme \"uber Ver\"anderungen von Zeichenreihen nach gegebener Regeln. Kra. Vidensk. Selsk. Skrifter., 10, 1914.
(In: Selected Math. Papers, Univ. Forlaget, Oslo, 1977, pp. 493--524.)
\end{thebibliography}
\end{document}